\documentclass[10pt]{article}
\usepackage{float,amsmath,amsthm,amssymb,
	latexsym,graphicx,amsfonts,amstext,amscd,
	amsbsy,bezier,enumerate,bigints,subfigure,multicol,graphicx}
\theoremstyle{plain}

\newtheorem{Theorem}{Theorem}
\newtheorem{Remark}{Remark}

\newtheorem{Lemma}{Lemma}

\renewcommand{\sec}{{\rm sech}}

\title{{\bf\Large Stability of equilibrium solutions of a double power reaction-diffusion equation with a Dirac interaction}}
\author{{\bf\large C\'esar A. Hern\'andez M.}\footnote{Email: cahmelo@uem.br, Tel. (55)-44-3011-5358, Fax: (55)-11-3011-3873 }\hspace{2mm}
	{\bf\large}\vspace{1mm}\\
	{\it\small Department of Mathematics,  DMA-UEM}\\
	{\it\small Av. Colombo, 5790 Jd. Universit\'ario,,}\\
	{\it\small  CEP 87020-900, Maring\'a, PR, Brazil}
	\vspace{3mm}\\
	{\bf\large Edgar Mayorga L.}\footnote{Email:  edgar.mayorga@unisabana.edu.co, Tel. (571)-861 5555-25041}\hspace{2mm}
	{\bf\large}\vspace{1mm}\\
	{\it\small University of Sabana}\\
	{\it\small Department of Mathematics, Physics and Statistics}\\
	{\it\small Ch\'ia, Colombia}\vspace{3mm}}

\date{July 24, 2017}
\begin{document}
	\maketitle
	\begin{abstract}
		In this paper we provide detailed information about the instability of equilibrium solutions of a nonlinear family of localized reaction-difussion equations in dimensione one. Beyond we provide explicit formulas to the equilibrium solutions, via perturbation method and we calculate the exact number of positive eigenvalues of the linear operator associated to the stability problem, which allow us to compute the dimension of the unstable manifold. 
	\end{abstract}
	\textbf{Mathematics  Subject  Classification (2010)}. Primary
	35K05, 35B10; 35B35; 35B38.\\
	\textbf{Key  words}. Reaction-difussion equation, Dirac interaction, Stability of equilibrium solutions, Blow-up of solutions, Analytic perturbation.

\section{Introduction}\label{int}

In this paper, we study the stability/instability of equilibrium solutions associated to the following generalized Huxley equation with a point defect interaction (henceforth GDH), 
\begin{equation}\label{deltalog}
u_{t}-u_{xx}=Z\delta(x)u+wu+au^{p}+bu^{2p-1},\hspace{0.7cm}\text{for }t,x\in\mathbb{R},
\end{equation}
where $u=u(x,t)$, $\delta:H^1(\mathbb{R})\rightarrow \mathbb{R}$ defined by $\delta(g)=g(0)$ the Dirac distribution localized at zero, and $a,b,p,Z$ are real parameters with $p>1$.\\

The GDH equation has many applications. For instance, when $Z=0$, $w=-\gamma_1\beta_1$, $a=\beta_1(1+\gamma_1)$ and $b=-\beta_1$, with $\beta_1>0$ and $\gamma_1\in(0,1)$, the GDH equation is reduced to the genearalized Huxley equation, namely
\begin{equation}\label{Huxley}
u_{t}-u_{xx}=\beta_1u(1-u^{p-1})(u^{p-1}-\gamma_1),
\end{equation}
which describes nerve pulse propagation in nerve fibres \cite{Aronson}, wall motion in liquid crystals \cite{Wang}, genetic population \cite{Lou} and  combustion \cite{Gilding}.
In chemistry, when $Z\neq 0$, the GDH equation can be considered as an specific model describing the concentration $u$ of a substance distributed one dimension space under the influence of local reaction site, bulk reaction and trasport, 
see equation (II.1) in \cite{Bimpong} and the references therein for details.\\ 

From the mathematical point of view, exact travelling solitary wave solutions, exact equilibrium solutions, and numerical solutions of the equation (\ref{Huxley}) have been discussed in the last years, see (\cite{Liu}, \cite{Wang}, \cite{Yefimova}, \cite{Deng}, \cite{Gao}, \cite{NewWang}, \cite{Batiha}, \cite{Talaat}, \cite{Arora}). The stability/instability of travelling wave solutions and equilibrium solutions of more general equations than the equation (\ref{Huxley}) have been discussed widely in \S 5.4 of \cite{DAN HENRY}, and more recently in \cite{Ghazaryan}. The problem of blow-up of solutions of semilinear parabolic equations have been also discussed in the las decades, for a survey on this subject, we refer the reader to \cite{Vazquez}.\\

However, the existence and stability of equilibrium solutions as well as the blow up of  solutions of the GDH equation, when $Z\neq 0$, have not been studied yet.\\

\textbf{A-} \textit{Equilibrium solutions of the GDH equation.} \\\\
By an equilibrium solution of the equation (\ref{deltalog}), we mean a function $\phi$ in the domain of the operator $\partial_{xx}+Z\delta(x)$, that is to say
\begin{equation*}\label{Dain}
\phi\in{\cal D}(\partial_{xx}+Z\delta(x))=\left\{g\in H^1(\mathbb{R})\cap H^2(\mathbb{R}\setminus\left\{0\right\})|g'(0+)-g'(0-)=-Zg(0)\right\}
\end{equation*}
satisfying the differential equation
\begin{equation}\label{ltaordenada}
(\partial_{xx}+Z\delta(x))\phi+w\phi+a\phi^{p}+b\phi^{2p-1}=0.
\end{equation}
Now, if we set $\alpha=\frac{a}{p+1}$, $\beta=\frac{b}{p}$, and
\begin{equation}\label{sol}
\phi_{w,Z}(x)=\left[\frac{\alpha}{-w}+\frac{\sqrt{\alpha^2-\beta w}}{-w}\text{ cosh}\left((p-1)\sqrt{-w}\left(|x|+R^{-1}\left(\frac{Z}{2\sqrt{-w}}\right)\right)\right)\right]^{\frac{-1}{p-1}},
\end{equation}
with $R:(-\infty,\infty)\rightarrow (-1,1)$ being the diffeomorphism defined by
\begin{equation}\label{dif}
R(s)=\frac{\sqrt{\alpha^2-\beta w}\text{ senh}((p-1)\sqrt{-w}s)}{\alpha+\sqrt{\alpha^2-\beta w}\text{ cosh}((p-1)\sqrt{-w}s)}, 
\end{equation}
then, in the section \ref{equilibrios} below, we will show that
\begin{enumerate}
\item\label{um1} If $a>0$, and $b>0$, then the family of functions (\ref{sol}) are equilibrium solutions of the GDH equation, providing 
$$
\frac{Z^2}{4}<-w.
$$ 
\item\label{dois2} If $a>0$, and $b<0$, then the family of functions $\phi_{w,Z}$ are equilibrium solutions of the GDH equation, providing 
$$
\frac{Z^2}{4}<-w<-\frac{pa^2}{(p+1)^2b}.
$$ 
\end{enumerate} 
\begin{Remark} When $Z=0$, $p=2$, $w=-\gamma_1$, $a=1+\gamma_1$, $b=-1$, the solution $\phi_{w,Z}$ is the equilibrium solution of the Nagumo equation, see formula $(7)$ in \textnormal{\cite{Nagumo}}. 
When $Z=0$, $p=3$, $a=1$, $b=1$, $\phi_{w,Z}$ is the equilibrium solution appearing in the formula $(32)$ of \textnormal{\cite{Hernandez}}.     
\end{Remark}

\textbf{B-} \textit{Stability/Instability of the equilibrium solutions of GDH equation.}\\\\
The equilibrium solution $\phi_{w,Z}$ given in (\ref{sol}), (\ref{dif}) is \emph{stable} in $H^1(\mathbb{R})$ by the flow of the GDH equation (\ref{deltalog}), if for every $\epsilon>0$ there exists $\delta>0$ such that, 
$$
\text{if}\hspace{0.3cm}g\in H^1(\mathbb{R})\hspace{0.3cm}\text{with}\hspace{0.3cm}||g-\phi_{w,Z}||_1<\delta,\hspace{0.3cm}\text{then}\hspace{0.3cm}||u(t)-\phi_{w,Z}||_1<\epsilon, 
$$
for all $t>0$, here $u(t)$ denotes the solution of the equation (\ref{deltalog}) generated by the initial data $u(0)=g$. Otherwise, $\phi_{w,Z}$ is \emph{unstable}.\\

We are now in position to establish the principal result of this note.

\begin{Theorem}\label{MainResult} Let $a,b,p,Z,w$ parameters satisfying the conditions \ref{um1}-\ref{dois2} above, then the equilibrium solutions $\phi_{w,Z}$ are unstable. In addition,
\begin{enumerate}
\item For $Z<0$, the unstable manifold associated to $\phi_{w,Z}$ has dimension $2$.
\item For $Z> 0$, the unstable manifold associated to $\phi_{w,Z}$ has dimension $1$.
\end{enumerate}
\end{Theorem}
The proof of the Theorem \ref{MainResult} can be obtained in the classical way, by analysing the spectrum of the linear self-adjoint operator 
\begin{equation}\label{L11,L223delta}
\left\{
\begin{aligned}
&{\cal L}_{Z}:{\cal D}(\partial_{xx}+Z\delta(x))\rightarrow L^2(\mathbb{R})\\
&{\cal L}_{Z}g=(\partial_{xx}+Z\delta(x))g+wg+ap\phi_{w,Z}^{p-1}g+b(2p-1)\phi_{w,Z}^{2p-2}g, 
\end{aligned}
\right.
\end{equation}
which is the linear approximation of the function
$$
G(u)=u_{xx}+Z\delta(x)u+wu+au^p+bu^{2p-1},
$$
at $\phi_{w,Z}$, i.e., $G'(\phi_{w,Z})g={\cal L}_{Z}g$, see Theorem 5.1.1, 5.1.3, and 5.2.1 in \cite{DAN HENRY}.
 In general, to count the number of positive eigenvalues of a linear operator is a delicate issue. In the case of the self-adjoint operator $\mathcal L_{Z}$ in (\ref{L11,L223delta}) our strategy is based in two basic facts. First, if one is $Z=0$, the spectrum of the self-adjoint operator $\mathcal L_0$ defined by 
\begin{equation*}\label{specA8}
\mathcal L_{0}=\frac{d^2}{dx^2}+w+a p\phi^{p-1}_{w, 0} +b(2p-1)\phi^{2p-2}_{w, 0}
\end{equation*}
with domain $H^2(\mathbb R)$ is well-known: there is only one positive eigenvalue which is simple, zero is a simple eigenvalue  with eigenfunction $\frac{d}{dx}\phi_{\omega,0}$. The rest of the spectrum is negative and away from zero. Second, if $Z$ is small, $\mathcal L_{ Z}$ can be considered as a {\it real-holomorphic perturbation} of $\mathcal L_{0}$. So, we have that the spectrum of $\mathcal L_{Z}$ depends holomorphically on the spectrum of  $\mathcal L_{0}$. Then we obtain that for $Z<0$ there are exactly two positive  eigenvalues of $\mathcal L_{Z}$ and exactly one for $Z>0$. We refer the reader to Section \ref{spectral} for the precise details on these statements.
 
\begin{Remark} 
When $Z=0$, the proof of the instability of the equilibrium solution $\phi_{w,0}$ can be obtained as a consequence of the  
Theorem \textnormal{5.1.3} in \textnormal{\cite{DAN HENRY}} and from the spectral properties of the operator ${\cal L}_0$ presented for instance in Theorem \textnormal{B.61} in \textnormal{\cite{Angulopavabook}}. 
\end{Remark}

This paper is organized as follows. In section \ref{boacolh}, we establish a local and global well-posedness theory for the GDH equation, in addition we establish the existence of solutions that blow up in finite time. Section \ref{equilibrios} describes the construction of the profile $\phi_{\omega, Z}$ in (\ref{sol}) for $w,Z,a, b$ satisfying the conditions in the Theorem \ref{solution10} below. Section \ref{spectral} describes the spectral theory for the operators $\mathcal L_{Z}$ in (\ref{L11,L223delta}).


\section[Local and global well posedness]{Local and global well-posedness for the GDH equation}\label{boacolh}

In this section we discuss some results about the local and global well-posedness problem associated to the GDH equation in $H^1(\mathbb{R})$
\begin{equation}\label{cachy12}
\left\{
\begin{array}{lll}
\displaystyle u_{t}+(A_Z-w)u=au^p+bu^{2p-1},\\
u(0)=g\in H^1(\mathbb{R}),
\end{array}
\right.
\end{equation}
where
\begin{equation}\label{FormalExpression}
A_Z:=-\frac{d^2}{dx^2}-Z\delta(x).
\end{equation}
Since our approach will be based in the abstract results in \S 3 of \cite{DAN HENRY}, we will establish the necessary framework. Initially, we recall that the formal expression in (\ref{FormalExpression}) can be understood as the family of self-adjoint operators with domain
\begin{equation*}\label{dain23}
D(A_Z)=\left\{g\in H^1(\mathbb{R})\cap H^2(\mathbb{R}-\left\{0\right\})|g'(0+)-g'(0-)=-Zg(0) \right\},
\end{equation*}
which represent all the self-adjoint extensions associated to  the following closed, symmetric, densely defined linear operator (see \cite{Albeverio}): 
\begin{equation*}
\left\{
\begin{aligned}
A_0&=-\frac{d^2}{dx^2}\\
D(A_0)&=\{g\in H^2(\mathbb{R}): g(0)=0 \}.
\end{aligned}
\right.
\end{equation*}
Moreover, for $Z\in \mathbb R$ we have that the essential spectrum of $A_Z $ is the nonnegative real axis, $\Sigma_{ess}(A_Z)=[0,+\infty)$. For $Z>0$, $A_Z $ has exactly one negative, simple eigenvalue, i.e., its discrete spectrum $\Sigma_{dis}(A_Z )$ is $\Sigma_{dis}(A_Z )=\{{-Z^2/4}\}$, with a strictly (normalized) eigenfunction $
 \Psi_Z(x)=\sqrt{\frac{Z}{2}}e^{-\frac{Z}{2}|x|}$. For $Z\leqq 0$,  $A_Z $ has not discrete spectrum, $\Sigma_{dis}(A_Z )=\emptyset$. Therefore the 
operators $A_Z$ are bounded from below,
\begin{equation}\label{boundbelo}
\left\{
\begin{aligned}
&A_Z\geq-Z^2/4, \hspace{0.5cm}& Z>0 \ \\ 
&A_Z\geq 0 & Z\leq0 .
\end{aligned}
\right.
\end{equation}

\begin{Theorem}\label{cazi} For any $u_0\in H^1(\mathbb{R})$, there exists $T>0$ and a unique solution $u$ of  $(\ref{cachy12})$ such that $u\in C^1([0,T); H^{1} (\mathbb{R}))$ and  $u(0)=u_0$.  For each $T_0\in (0,T)$ the mapping
$$
u_0\in H^1(\mathbb{R}) \to u\in C([0,T); H^1(\mathbb{R}))
$$
is continuous. If an initial data $u_0$ is even the solution $u(t)$ is also even.
\end{Theorem}
\begin{proof}
The proof of this theorem is an application of Theorem 3.3.3 in \cite{DAN HENRY}. First of all, from (\ref{boundbelo}), we have the self-adjoint operator ${\cal A}\equiv A_{Z}-w+a_1$ on the space $X=L^2(\mathbb{R})$, with $a_1 >Z^2/4+w$ for $Z>0$ and  $a_1>w$ for $Z\leqq 0$, and domain $D({\cal A})=D(A_{Z})$, satisfies $\Sigma({\cal A})> 0$. Secondly, in our case it is possible to consider the space $X^{1/2}=H^1(\mathbb{R})$ with norm 
\[
||u||^2_{1/2}=||u_x||_{L^2}^2+(-w+a_1)||u||_{L^2}^2-Z|u(0)|^2,
\] 
which is equivalent to the usual norm in $H^1(\mathbb{R})$. Lastly, it is well known that the function $u\in H^1(\mathbb{R})\rightarrow f(u)=au^{p}+bu^{2p-1}$ is locally Lipschitzian.     
\end{proof}

Now, for the case $Z=0$ in (\ref{cachy12}) is well-known that the double-power nonlinearity induce restrictions on the existence of global solutions. The following theorem shows that a similar picture happens for $Z\neq 0$.

\begin{Theorem}\label{gwpdel} i) For $Z>0$, and $\frac{Z^2}{4}<-w$. The solution of the Cauchy problem $(\ref{cachy12})$ is globally well defined in $H^1(\mathbb{R})$ providing the norm of the initial data $u(0)=g$ small in $H^{1}(\mathbb{R})$.\\
ii) For $Z\leq 0$, and $0<-w$. The solution of the Cauchy problem $(\ref{cachy12})$ is globally well defined in $H^1(\mathbb{R})$ providing the norm of the initial data $u(0)=g$ small in $H^{1}(\mathbb{R})$.
\end{Theorem}
\begin{proof} This result is a consequence of the stability of the equilibrium solution $u\equiv 0$ in $H^{1}(\mathbb{R})$ that can be obtained via spectral analysis of the 
linear operator $A\equiv-A_Z+w$, and via Theorem 5.1.1 in \cite{DAN HENRY}. Indeed, from (\ref{boundbelo}), we deduce that  
\begin{enumerate}[i)]
\item For $Z>0$, $\Sigma(-A_Z+w)=(-\infty,w]\cup\{w+\frac{Z^2}{4}\}$. 
\item For $Z\leq 0$, $\Sigma(-A_Z+w)=(-\infty,w]$. 
\end{enumerate}
Then the spectrum of the operator $A$ is negative and away from zero. The rest of the hypothesis of the Theorem 5.1.1 were discussed above.
\end{proof}
Now, let $u$ be the solution of the Cauchy problem (\ref{cachy12}), let $Z>0$, and consider $t\to R(t)\in\mathbb{R}$ the function defined by the expression     
\begin{equation}\label{erre}
R(t)=\int_{-\infty}^{\infty} u(x,t)e^{\frac{-Z|x|}{2}} dx.
\end{equation}
In addition, consider 
$$
\lambda=w+Z^2/4<0,\hspace{0.4cm}\beta=a(Z/4)^{\frac{p-1}{p}},\hspace{0.4cm}\gamma=b(Z/4)^{\frac{2(p-1)}{2p-1}},\hspace{0.4cm}\text{and}\hspace{0.4cm}   
z_1=\frac{-\beta\lambda^{-1}-\sqrt{(\beta\lambda^{-1})^2-4\gamma\lambda^{-1}}}{2}.
$$
Then, the result below establishes the blow up of solutions of the Cauchy problem $(\ref{cachy12})$ for specific values of the parameters.

\begin{Theorem}\label{blow} For $Z>0$, $a,b>0$ and $\frac{Z^2}{4}<-w$. The solution of the Cauchy problem \textnormal{(\ref{cachy12})} with initial positive data $u(0)=g$ blows up in finite time providing 
$$
\int_{-\infty}^{\infty}g(x)e^{-\frac{Z}{2}|x|}dx>R_1,\hspace{1.0cm}\text{with}\hspace{1.0cm}R_1=\left[\frac{z_1\lambda}{\gamma}\right]^{\frac{1}{p-1}}.
$$ 
Moreover, if $T$ denotes the time where the solution blows up, we have that
$$
T\leq\int_{R(0)}^{\infty}\frac{1}{\lambda s+\beta s^p+\gamma s^{2p-1}}ds. 
$$
\end{Theorem}
\begin{proof}
Let $\phi(x)=e^{-\frac{Z|x|}{2}}$, from (\ref{erre}), we obtain
\begin{equation*}\label{Estimat123}
\begin{aligned}
R'(t)&=\int_{-\infty}^{\infty} u_t\phi(x)dx
=\int_{-\infty}^{\infty} [u_{xx}+Z\delta(x)+wu+au^p+bu^{2p-1}]\phi(x)dx.
\end{aligned}
\end{equation*}
Now, since $Z^2/4$ is an eigenvalue of the self-adjoint operator $-A_{Z}$, with associated eigenfunction $\phi$, then we have 
\begin{equation}\label{dest}
\int_{-\infty}^{\infty} [-A_Zu+wu]\phi(x)dx=\left(\frac{Z^2}{4}+w\right)\int_{-\infty}^{\infty} u\phi(x)dx.
\end{equation}
On the other hand, Holder inequality and the positivity of the solution $u$ imply that
\begin{equation}\label{Holde123}
\int_{-\infty}^{\infty} [au^p+bu^{2p-1}]\phi(x)dx\geq a||\phi||_{L^1}^{\frac{1-p}{p}}R^{p}+b||\phi||_{L^1}^{\frac{2(1-p)}{2p-1}}R^{2p-1}.
\end{equation}
Since $||\phi||_{L^1}=4/Z$, from  (\ref{dest}), (\ref{Holde123}) we get the following differential inequality
\begin{equation}\label{Estim123}
\lambda R(t)+\beta R^p(t)+\gamma R^{2p-1}(t)\leq R'(t).
\end{equation} 
By standard arguments, it is possible to prove that if $R(0)>R_1$, with $R_1$ being the unique positive constant solution of (\ref{Estim123}), then, there exists $T>0$, such that,
\begin{equation}\label{e123}
T\leq\int_{R(0)}^{\infty}\frac{1}{\lambda s+\beta s^p+\gamma s^{2p-1}}ds\hspace{1.0cm}\text{and}\hspace{1.0cm}\lim_{t\to T}R(t)=\infty. 
\end{equation}
where, $h(s)=\lambda s+\beta s^p+\gamma s^{2p-1}$. Now, since
\begin{equation}\label{Inequ123}
R(t)\leq ||u(t)||_{L^{2}}||\phi||_{L^{2}}=||u(t)||_{L^{2}}\left[\frac{2}{Z}\right]^{\frac{1}{2}},
\end{equation}
from (\ref{e123}) and (\ref{Inequ123}), we conclude
$$
\lim_{t\to T}||u(t)||_{L^{2}}=\infty.
$$
\end{proof}
\begin{Remark}
Following the ideas and notations of Theorem \ref{blow}, it is possible to prove that for $Z>0$, $a,b>0$, and $\frac{Z^2}{4}\geq-w$, the solution of the Cauchy problem \textnormal{(\ref{cachy12})} with initial positive data $u(0)=g$ blows up in finite time. 
\end{Remark}
  
\begin{Remark} Let $Z>0$, $a,b>0$, $\frac{Z^2}{4}<-w$, and $\phi_{w,Z}$ given in \textnormal{(\ref{sol})} the positive equilibrium solution of the Cauchy problem \textnormal{(\ref{cachy12})} with $g\equiv \phi_{w,Z}$. Then, as a consequence of the Theorem \textnormal{\ref{blow}}, we conclude that
$$
\int_{-\infty}^{\infty}\phi_{w,Z}(x)e^{-\frac{Z}{2}|x|}dx\leq R_1,\hspace{1.0cm}\text{with}\hspace{1.0cm}R_1=\left[\frac{z_1\lambda}{\gamma}\right]^{\frac{1}{p-1}}.
$$
\end{Remark}


\section{Equilibrium solutions of the GDH equation}\label{equilibrios} 
In this section, we deal with the deduction of the explicit equilibrium solutions given in (\ref{sol}), (\ref{dif}) for the GDH equation (\ref{deltalog}) with $a>0$, $b\neq 0$.
We consider the cases, $Z=0$ and $Z\neq 0$. For $Z=0$ in (\ref{ltaordenada}), we have that $\phi\equiv \phi_w$ satisfies the nonlinear elliptic equation
\begin{equation}\label{ordendoisz0}
\phi''+w\phi+a\phi^{p}+b\phi^{2p-1}=0.
\end{equation}
The quadrature method and the boundary condition of the function $\phi\rightarrow 0$ as $|x|\rightarrow \infty$, imply that
\begin{equation}\label{ord1z0}
[\phi']^2+w\phi^2+2\alpha\phi^{p+1}+\beta\phi^{2p}=0,
\end{equation}
where, $\alpha=a/(p+1)$, $\beta=b/p$. 
In order to obtain an explicit solution of the equation (\ref{ordendoisz0}), we will assume that $w<0$ and $0<\phi$. Replacing 
$$
y=\phi^{p-1},
$$ 
in (\ref{ord1z0}), we obtain
\begin{equation}\label{inttrans}
\frac{1}{(p-1)\sqrt{-w}}\bigintssss{\frac{dy}{y\sqrt{1+2\alpha w^{-1}y+\beta w^{-1}y^2}}}=x
\end{equation}
so, by assuming $c$ a positive constant, we have the formula
\begin{equation*}\label{inttransit}
\bigintssss{\frac{dy}{y\sqrt{1+2\alpha w^{-1}y+\beta w^{-1}y^2}}}=-\ln\left[c\left(\frac{\alpha y+w}{-y}+\frac{\sqrt{\beta wy^2+2w\alpha y +w^2}}{y}\right)\right].
\end{equation*}
Now, for $c=1/\sqrt{\alpha^2-\beta w}$, and recalling that
$$
\text{arccosh}(x)=\ln(x-\sqrt{x^2-1}),\hspace{0.5cm}\text{for all}\hspace{0.5cm}x\geq 1,  
$$ 
we can rewrite the integral in (\ref{inttrans}) as 
\begin{equation*}\label{inttransdo}
\bigintssss{\frac{dy}{y\sqrt{1+2\alpha w^{-1}y+\beta w^{-1}y^2}}}=-\text{arccosh}\left(\frac{\alpha y+w}{-y\sqrt{\alpha^2-\beta w}}\right).
\end{equation*}
Note that a solution of the equation (\ref{ordendoisz0}) is given implicitly by the formula
$$
-\text{arccosh}\left(\frac{\alpha \phi^{p-1}+w}{-\phi^{p-1}\sqrt{\alpha^2-\beta w}}\right)=(p-1)\sqrt{-w}x,
$$
or
\begin{equation}\label{numerator}
\phi(x)=\left[\frac{-w}{\alpha+\sqrt{\alpha^2-\beta w}\text{ cosh}((p-1)\sqrt{-w}x}\right]^{\frac{1}{p-1}}=\left[\frac{\alpha}{-w}+\frac{\sqrt{\alpha^2-\beta w}}{-w}\text{ cosh}\left((p-1)\sqrt{-w}x\right)\right]^{\frac{-1}{p-1}}
\end{equation} 
with $-w>0$ and $\alpha^2-\beta w>0$. Now, we notice that if $b>0$, the solution $\phi$ in (\ref{numerator}) is well defined for all $w<0$. In contrast, if $b<0$, the solution $\phi$ in (\ref{numerator}) is well defined for $w$ satisfying the following condition 
$$
0<-w<-\frac{pa^2}{(p+1)^2b}.
$$
Now, we proceed to calculate solutions of the equation (\ref{ltaordenada}) when $Z\neq 0$. 
The following lemma shows us some of the properties that a solution $\phi\in H^1(\mathbb{R})$ of the equation (\ref{ltaordenada}) must satisfy.
\begin{Lemma}\label{regul} Let $\phi\in H^1(\mathbb{R})$ be a solution of \textnormal{(\ref{ltaordenada})}, then $\phi$ satisfies the following properties,
\begin{subequations}\label{initialvalue3}
\begin{align}
&\phi\in C^j(\mathbb{R}-\{0\})\cap C(\mathbb{R}),\text{ }\text{ } j=1,2.\label{reg}\\
&\phi''(x)+w\phi(x)+a\phi^{p}(x)+b\phi^{2p-1}(x)=0,\hspace{0.6cm}\text{for all}\hspace{0.6cm} x\neq 0. \label{eqdifx}\\
&\phi'(0+)-\phi'(0-)=-Z\phi(0).\label{condsal}\\
&\phi'(x),\phi(x)\rightarrow 0, \hspace{0.6cm}\text{if}\hspace{0.3cm} |x|\rightarrow\infty.\label{complimitado}
\end{align}
\end{subequations}
\end{Lemma}
\begin{proof}
The proof of this lemma follows the ideas of the proof of Lemma 3.1 in \cite{jeanjean}. The properties (\ref{reg}) and (\ref{complimitado}) are proved by a standard boostrap argument, namely, for all $\xi\in C_0^{\infty}(\mathbb{R}\setminus\{0\})$, the function $\xi\phi$ satisfies 
$$
(\xi\phi)''+w(\xi\phi)=\xi''\phi+2\xi'\phi'-a\xi\phi^p-b\xi\phi^{2p-1}
$$
in the sense of distributions. Since the right hand side of the previous identity is in $L^2(\mathbb{R})$, then $\xi\phi\in H^2(\mathbb{R})$, that is to say, $\phi\in H^2(\mathbb{R}\setminus \{0\})\cap C^1(\mathbb{R}\setminus\{0\})$. The equation (\ref{eqdifx}) follows from the fact that $C_0^{\infty}(\mathbb{R}\setminus\{0\})$ is dense in $L^2(\mathbb{R})$. In relation to (\ref{condsal}), it is enough to ``integrate'' (\ref{ltaordenada}) from $-\epsilon$ to $\epsilon$,
$$
\int_{-\epsilon}^{\epsilon}\phi'' dx+w \int_{-\epsilon}^{\epsilon}\phi dx +Z\int_{-\epsilon}^{\epsilon}\delta(x)\phi dx+\int_{-\epsilon}^{\epsilon}a\phi^p+b\phi^{2p-1} dx=0.
$$
If $\epsilon\to 0$, we obtain that $\phi'(0+)-\phi'(0-)=-Z\phi(0)$.
\end{proof}

We notice that the function 
\begin{equation}\label{esest}
\phi_{s}(x):=\phi(|x|-s),\hspace{0.5cm}s\in\mathbb{R}, 
\end{equation}
where $\phi$ given in (\ref{numerator}), satisfies all of the properties of the previous lemma except possibly the jump condition (\ref{condsal}). Since $\phi_s$ is an even function the condition (\ref{condsal}) can be rewritten as
\begin{equation*}\label{jump}
\phi_s'(0+)=-\frac{Z}{2}\phi_s(0),\hspace{0.5cm}\text{or equivalently,}\hspace{0.5cm}\phi'(s)=\frac{Z}{2}\phi(s).
\end{equation*} 
From (\ref{numerator}), we obtain that
\begin{equation}\label{translat}
\frac{\sqrt{\alpha^2-\beta w}\text{ senh}((p-1)\sqrt{-w}s)}{\alpha+\sqrt{\alpha^2-\beta w}\text{ cosh}((p-1)\sqrt{-w}s)}=\frac{-Z}{2\sqrt{-w}}.
\end{equation}  
If we define $R:(-\infty,\infty)\rightarrow (-1,1)$ by
\begin{equation}\label{Erre}
R(s)=\frac{\sqrt{\alpha^2-\beta w}\text{ senh}((p-1)\sqrt{-w}s)}{\alpha+\sqrt{\alpha^2-\beta w}\text{ cosh}((p-1)\sqrt{-w}s)}. 
\end{equation}
We have that $R$ is an odd, increassing diffeomosphism between the intervals $(-\infty,\infty)$ and (-1,1). In particular, from the expression (\ref{translat}), we conclude that 
$$
\frac{Z^2}{4}<-w,
$$
and furthermore, that
\begin{equation}\label{ese}
s=R^{-1}\left(\frac{-Z}{2\sqrt{-w}}\right).
\end{equation}
Finally, from (\ref{numerator}), (\ref{esest}) and (\ref{ese}), we can conclude that for $\alpha=\frac{a}{p+1}$ and  $\beta=\frac{b}{p}$, the function $\phi_{w,Z}$ given by
\begin{equation}\label{numeratorwithZ}
\phi_{w,Z}(x)=\left[\frac{\alpha}{-w}+\frac{\sqrt{\alpha^2-\beta w}}{-w}\text{ cosh}\left((p-1)\sqrt{-w}\left(|x|+R^{-1}\left(\frac{Z}{2\sqrt{-w}}\right)\right)\right)\right]^{\frac{-1}{p-1}}
\end{equation}
is a solution of the equation (\ref{ltaordenada}) providing 
\begin{enumerate}[i)]
\item For $a>0$, $b> 0$: $\frac{Z^2}{4}<-w$. 
\item For $a>0$, $b<0$: $\frac{Z^2}{4}<-w<-\frac{pa^2}{(p+1)^2b}$.
\end{enumerate}
We observe that if $Z=0$ in the previous formula, we recover the function $\phi$ given in (\ref{numerator}), namely $\phi_{w,0}=\phi$. 
Thus, we can establish an existence result of equilibrium solutions for (\ref{ltaordenada}).
\begin{Theorem}\label{solution10} i) For $a>0$, $b>0$ and $\frac{Z^2}{4}<-w$, the family of functions $\phi_{w,Z}$ given in $(\ref{numeratorwithZ})$ are equilibrium solutions of the equation \textnormal{(\ref{deltalog})}. Moreover, the mapping $Z\in(-2\sqrt{-w},2\sqrt{-w})\rightarrow\phi_{w,Z}$ is a real analytic function.\\\\
ii) For $a>0$, $b<0$ and $\frac{Z^2}{4}<-w<-\frac{pa^2}{(p+1)^2b}$, the family of functions $\phi_{w,Z}$  given in \textnormal{(\ref{numeratorwithZ})} are equilibrium solutions of the equation \textnormal{(\ref{deltalog})}. Moreover, the mapping $Z\in(-2\sqrt{-w},2\sqrt{-w})\rightarrow\phi_{w,Z}$ is a real analytic function.  
\end{Theorem}
Figure 1 shows the profile of $\phi_{\omega, Z}$ in (\ref{numeratorwithZ}) in  the case $p=5$, $a=6$ $b=-1$.
\begin{figure}[htb]
  \centering
  \subfigure[$\phi_{w,Z}$: $w=-4$, $Z=-2$.]{
      \includegraphics[scale=0.3]{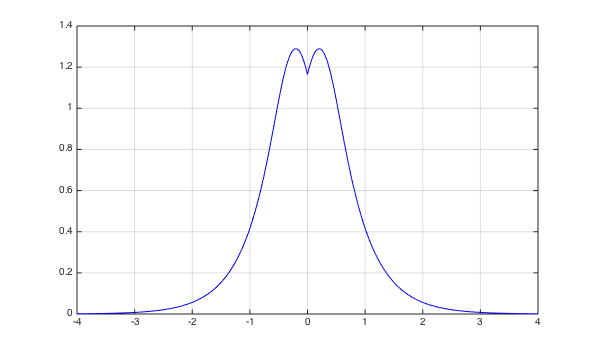}}
      \qquad
	\subfigure[$\phi_{w,Z}$: $w=-4$, $Z=2$.]{
      \includegraphics[scale=0.3]{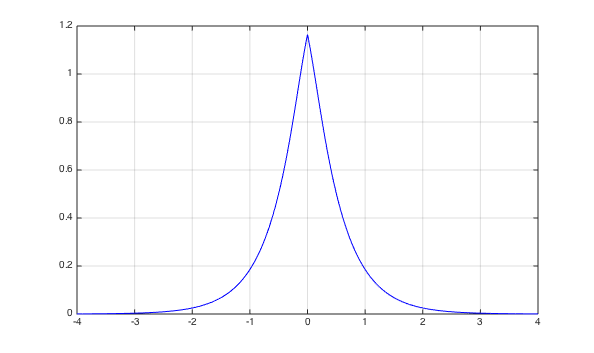}}
  \qquad
	\subfigure[$\phi_{w,Z}$: $w=-4$, $Z=0$.]{
      \includegraphics[scale=0.3]{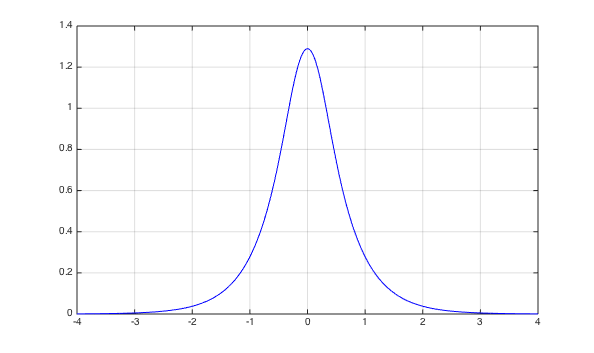}}
  \label{gras} 
	\caption{The equilibrium solutions $\phi_{w,Z}$ for $Z<0$, $Z=0$ and $Z>0$.}
   \end{figure}

\begin{Remark}
Considering $b=0$, $a>0$, and $\frac{Z^2}{4}<-w$, the function $\phi_{w,Z}$ in \textnormal{(\ref{numeratorwithZ})} can be written in the following form:
$$
h_{w,Z}(x)=\left[\frac{(p+1)(-w)}{2a}\sec^2\left(\frac{(p-1)\sqrt{-w}}{2}|x|+\tanh^{-1}\left(\frac{Z}{2\sqrt{-w}}\right) \right)\right]^{\frac{1}{p-1}},
$$
it is an equilibrium solution of the GDH equation, for $b=0$.\\
\end{Remark}
\section{Spectral properties of ${\cal L}_{Z}$}\label{spectral}

This section is devoted to the study of the spectral properties of the operators $\mathcal L_{Z}$ given in (\ref{L11,L223delta}). We establish the relation  between the second variation of the functional $S_{w,Z}:H^{1}(\mathbb{R})\to\mathbb{R}$ defined by
$$
S_{w,Z}(u)=\int\limits_{\mathbb{R}}\left(\frac{(u')^2}{2} -\frac{wu^2}{2} -\frac{au^{p+1}}{p+1}-\frac{bu^{2p}}{2p}\right)dx-\frac{Zu^2(0)}{2},
$$ 
at $\phi_{\omega, Z}$ and the self-adjoint operators $\mathcal L_{Z}$. It can be easily verified that the equilibrium solution $\phi_{w}=\phi_{w, Z}$  is a critical point of $S_{w, Z}$. Indeed, for $u,v\in H^1(\mathbb R)$,
 \begin{equation*}
 S_{w, Z}'(u)v=\left.\frac{d}{dt}S_{\omega, Z}(u+tv)\right|_{t=0}=\int\limits_{\mathbb{R}}u'v' \, dx-\int\limits_{\mathbb{R}}(wu+au^p+bu^{2p-1})v \, dx-Z u(0)v(0).
 \end{equation*}
Since $\phi_{w}$  satisfies \eqref{ltaordenada}, $S_{w,Z}'(\phi_{\omega})=0$. 
In the same way, we get 
\begin{equation}\label{q_form}
S_{w, Z}''(\phi_{w})(u_1,v_1)=\int\limits_{\mathbb{R}}u_1'v_1'dx-\int\limits_{\mathbb{R}}(w+ ap\phi^{p-1}_{w} + b(2p-1)\phi^{2p-2}_{w})u_1v_1 \, dx -Z u_1(0)v_1(0) 
\end{equation}
with $D(S_{w, Z}''(\phi_{w}))=H^1(\mathbb R)\times H^1(\mathbb R)$.
Note that the form $S_{w, Z}''(\phi_{w})$ is bilinear 
bounded from below and closed. Therefore, by  the First Representation Theorem (see  \cite[Chapter VI, \S 2.1]{Kato}), it defines an operator $-\mathcal L_{Z}$ such that
 \begin{equation}\label{opera}
 \left\{
 \begin{aligned}
 &D(\mathcal {-L}_Z)=\{v_1\in H^1(\mathbb R): \exists v_2\in   L^2(\mathbb{R})\; s.t.\; \forall z\in H^1(\mathbb R), \; S_{w, Z}''(\phi_{w})(v_1,z)=(v_2,z) \},\\ 
 &\mathcal {-L}_Z v_1=v_2.
 \end{aligned}
 \right.
 \end{equation}
\begin{Theorem}\label{repres}
The operator $-\mathcal L_{Z}$ determined in \textnormal{(\ref{opera})} is given by 
\begin{equation*}\label{L12}
-\mathcal L_{Z}=-\frac{d^2}{dx^2} -\omega -ap\phi^{p-1}_{w} - b(2p-1)\phi^{2p-2}_{w},
\end{equation*} 
on the domain $D_Z:= D\left(-{\cal L}_{Z}\right)=\left\{g\in H^1(\mathbb{R})\cap H^2(\mathbb{R}-\left\{0\right\})|g'(0+)-g'(0-)=-Zg(0) \right\}$.
\end{Theorem}
\begin{proof}
The proof of this theorem follows the same lines as in Le Coz {\it et al.} \cite{LeCoz}.
\end{proof}

Now, we proceed to establish some specific spectral properties of the operator $-{\cal L}_{Z}$. Here we will consider the parameters $a, b$, $w$ and $Z$ such that satisfy the relations in Theorem \ref{solution10}.
\begin{Theorem}\label{espl0} For $w<0$ and $\phi_{w,0}$ being the equilibrium solution $\phi_{w,Z}$ with $Z=0$, we consider the self-adjoint linear operator $-{\cal L}_0:H^2(\mathbb{R})\rightarrow L^2(\mathbb{R})$
$$
-{\cal L}_0=-\frac{d^2}{dx^2}-w-ap\phi^{p-1}_{w,0}-b(2p-1)\phi^{2p-2}_{w,0}.
$$
Then $-{\cal L}_0$ has a unique negative simple eigenvalue $-\lambda$, with $\lambda>0$. Zero is a simple eigenvalue with eigenfunction $\phi'$. The rest of the espectrum is away from zero and the essential spectrum is given by the interval $[-w,\infty)$.
\end{Theorem}
\begin{proof} This result is a consequence of the classical Sturm Liouville oscillation theory (see \cite{BerShu91}). 
\end{proof}
Now, we proceed to study the spectral properties of the operator ${\cal L}_{Z}$ for $Z\neq 0$.
\begin{Lemma}\label{Kert} For $Z$ satisfying the conditions in Theorem \textnormal{\ref{solution10}} and $Z\neq 0$, the kernel of ${\cal L}_{Z}$ is trivial.
\end{Lemma}
\begin{proof} Let $v\in Ker({\cal L}_{Z})$. It is clear that all elements in the kernel of the operator ${\cal L}_{Z}$ are solutions of 
\begin{equation}\label{valprop}
\left\{
\begin{aligned}
&{\cal L}_{Z}\phi(x)=0, \hspace{0.5cm} x>0.\\ 
&\phi\in L^2((0,\infty)).
\end{aligned}
\right.
\end{equation}
It is well known that the linear problem (\ref{valprop}) 
has dimension one. Since $\phi_{w,Z}'$ satisfies the problem (\ref{valprop}) then we conclude that there exists $\alpha\in\mathbb{R}$ such that $v(x)=\alpha\phi'_{w,Z}(x)$, for all $x>0$. A similar argument can be applied to the problem
\begin{equation*}\label{valpropiosi}
\left\{
\begin{aligned}
&{\cal L}_{Z}\phi(x)=0, \hspace{0.5cm}x<0.\\ 
&\phi\in L^2((-\infty,0)).
\end{aligned}
\right.
\end{equation*}
Thus, there exists a real number $\beta$ such that $v(x)=\beta\phi'_{w,Z}(x)$, for $x<0$. From the continuity of $v$ and the parity of the function $\phi_{w,Z}$, we deduce that $\alpha=-\beta$ so that $v$ is writen as
\begin{equation}\label{parada}
v(x)=
\begin{cases}
\alpha\phi'_{w,Z}(x), & \text{if}\hspace{0.3cm}  x\geq0,\\
-\alpha\phi'_{w,Z}(x) & \text{if}\hspace{0.3cm} x<0.
\end{cases}
\end{equation}
Since $v\in {\cal D}({\cal L}_{Z})$, it follows that $v'(0+)-v'(0-)=-Zv(0)$. From (\ref{parada})   
\begin{equation}\label{nose}
v'(0+)-v'(0-)=\alpha\phi_{w,Z}''(0+)+\alpha\phi_{w,Z}''(0-)=-Z\alpha\phi_{w,Z}'(0+).
\end{equation} Suppose that $\alpha\neq0$, from (\ref{eqdifx}) and (\ref{nose}), we obtain that $\phi_{w,Z}''(0+)=-Z/2\phi_{w,Z}'(0+)$. Multiplying the equation (\ref{eqdifx}) by $g'$ and integrating on the interval $(0,R)$, we get 
\begin{equation*}
-\frac{1}{2}(g'(R))^2+\frac{1}{2}(g'(0+))^2-F(g(R))+F(g(0+))=0
\end{equation*}where $F(s)=ws^2/2+as^{p+1}/(p+1)+bs^{2p}/(2p)$. If $R\rightarrow\infty$ equation (\ref{complimitado}) imply that
\begin{equation}\label{tomala}
\frac{1}{2}(g'(0+))^2+F(g(0+))=0.
\end{equation} Since $\phi_{w,Z}$ satisfies the equation (\ref{eqdifx}), from (\ref{tomala}), we get that $\frac{1}{2}(\phi_{w,Z}'(0+))^2=-F(\phi_{w,Z}(0))$. In addition, as $\phi_{w,Z}$ is an even function, we obtain that $\phi_{w,Z}'(0+)=\frac{-Z}{2}\phi_{w,Z}(0)$. Therefore, we deduce that $\phi_{w,Z}(0)$ is a positive zero of the following function 
\begin{equation}\label{poli}
P(s)=\frac{Z^2}{8}s^2+w\frac{s^2}{2}+a\frac{s^{p+1}}{p+1}+b\frac{s^{2p}}{2p}.
\end{equation} 
On the other hand, from the equation (\ref{ltaordenada}), we have that
\begin{equation*}
\lim_{x\rightarrow 0+} \phi_{w,Z}''(x)=\phi_{w,Z}''(0+)=-w\phi_{w,Z}(0)-a\phi^{p}_{w,Z}(0)-b\phi^{2p-1}_{w,Z}(0).
\end{equation*}
Since,
\begin{equation*}
\phi_{w,Z}''(0+)=\frac{Z^2}{4}\phi_{w,Z}(0),
\end{equation*}
it follows that $\phi_{w,Z}(0)$ is a positive zero of the function 
\begin{equation}\label{poliester}
R(s)=\frac{Z^2}{4}s+ws+as^p+bs^{2p-1}.
\end{equation} 
Since $s_0=\phi_{w,Z}(0)$ is a zero of both (\ref{poli}) and (\ref{poliester}), we deduce that
\begin{equation}\label{identi}
s_0^{p-1}=-\frac{ap}{b(p+1)}.
\end{equation}
Notice that $s_0=\phi_{w,Z}(0)$ is positive, then (\ref{identi}) gives us a contradiction if $a,b>0$. Now, in the case $a>0$, $b<0$, from (\ref{numerator}) and (\ref{numeratorwithZ}) the relation $\phi^{p-1}_{w,Z}(0)\leq \phi^{p-1}(0)$ is obtained. Since $-w<-\frac{pa^2}{(p+1)^2b}$ from (\ref{numerator}), we get
$$
\phi^{p-1}(0)=\frac{-w}{\alpha}\frac{1}{1+\sqrt{1-\beta w\alpha^{-2}}}<-\frac{ap}{b(p+1)}\frac{1}{1+\sqrt{1-\beta w\alpha^{-2}}}
$$
which is also a contradiction. Therefore, we conclude that $\alpha=0$ and then $v\equiv0$.  
\end{proof}

Now we show that the family of operators $-{\cal L}_{Z}$ depends analytically of the variable $Z$, where $Z$ satisfies the conditions of Theorem \ref{solution10}.
\begin{Lemma}\label{lem6.5} As a function of the variable $Z$, $\left\{-{\cal L}_{Z}\right\}_{Z\in\mathbb{R}}$ is a real analytic family of self-adjoint operators of type \textnormal{(B)} in the sense of Kato.
\end{Lemma}
\begin{proof}
By the Theorem VII-4.2 in \cite{Kato} and \cite{Reed}, it is enough to show that the family of bilinear forms  
$S_{w, Z}''(\phi_{w})$ defined in (\ref{q_form}) are real analytic of type (B). Namely
\begin{enumerate}
\item The domain $D\left(S_{w, Z}''(\phi_{w})\right)$ of the forms $S_{w, Z}''(\phi_{w})$ is independent of the parameter $Z$. In our case, $D\left(S_{w, Z}''(\phi_{w})\right)=H^1(\mathbb{R})$ for all $Z\in\mathbb{R}$.
\item For each $Z\in\mathbb{R}$, $S_{w, Z}''(\phi_{w})$ is closed and bounded from below. 
\item Since $\phi$ and $R$ given in (\ref{numerator}) and (\ref{Erre}) respectively are analytic functions then $\phi_{w,Z}$ is an analytic function. Thus, for each $v\in{\cal D}$ the function $Z\rightarrow S_{w, Z}''(\phi_{w})(v,v)$ is analytic.   
\end{enumerate}
\end{proof}

Next, we use the Kato-Rellich Theorem to prove some specific properties of the second eigenvalue and eigenfunction of the operator $-{\cal L}_{Z}$.

\begin{Lemma}\label{piomega}
There exist $Z_0>0$ and analytic functions $\Pi_2:(-Z_0,Z_0)\rightarrow \mathbb{R}$, $\Omega_2:(-Z_0,Z_0)\rightarrow L^2(\mathbb{R})$, such that 
\begin{enumerate}[$(i)$]
\item For each $Z\in(-Z_0,Z_0)$, $\Pi_2(Z)$ is the second eigenvalue of $-{\cal L}_{Z}$, which is simple and $\Omega_2(Z)$ its corresponding eigenfunction.
\item $\Pi_2(0)=0$ and $\Omega_2(0)=\phi'$, where $\phi$ is given in \textnormal{(\ref{numerator})}.
\item $Z_0$ can be chosen small enough such that the spectrum of $-{\cal L}_{Z}$ with $Z\in(-Z_0,Z_0)$ is greater than $0$ except by the first 
$2$ eigenvalues.
\end{enumerate}
\end{Lemma}
\begin{proof} The proof proceeds in several steps:
\begin{enumerate}[(1)]
\item There is a positive $M$ such that $\sigma\left(-{\cal L}_{Z}\right)\cap(-\infty,-M)=\emptyset$ for $Z\in [-a,a]$, and $a$ small enough.
\item From the Theorem \ref{espl0}, defining $\lambda_{1,0}=-\lambda$ and $\lambda_{2,0}=0$, we can separete the spectrum $\sigma\left(-{\cal L}_{0}\right)$ of $-{\cal L}_{0}$ into two parts $\sigma_0=\{-\lambda,0\}$  
and $\sigma_1=[-w,\infty)$ by a simple closed curve $\Gamma\subset\rho\left(-{\cal L}_0\right)$ such that $\sigma_0\subset \text{Int}(\Gamma)$ and $\sigma_1$ in its exterior; where, $\text{Int}(\Gamma)$ denotes the interior of $\Gamma$.
\item From Lemma \ref{lem6.5}, we have that $-{\cal L}_{Z}$ converges to $-{\cal L}_{0}$ as $Z\to 0$ in the generalized sense. So, from Theorem IV, \S 3.16 in \cite{Kato}, we have that $\Gamma\subset\rho\left(-{\cal L}_Z\right)$ for $Z\in(-\epsilon_1,\epsilon_1)$, $\epsilon_1$ small enough, and $\sigma\left(-{\cal L}_{Z}\right)$ is also separated by $\Gamma$ into two parts such that the part of $\sigma\left(-{\cal L}_{Z}\right)$ inside $\Gamma$ consists of a finite set of eigenvalues with total algebraic multiplicity $2$.
\item For $i=1,2$, and $\gamma$ small enough we define circles $\Gamma_i=\{z\in\mathbb{C}:|z-\lambda_{i,0}|=\gamma\}$ such that $\Gamma_1\cap\Gamma_2=\emptyset$ and $\Gamma_i$ is in the interior of $\Gamma$. Thus from the nondegeneracy of the eigenvalues $-\lambda_{i,0}$, we obtain that there exists $0<Z_1<\epsilon_1$ such that for $Z\in(-Z_1,Z_1)$, $\sigma\left(-{\cal L}_Z\right)\cap \text{Int}(\Gamma_i)=\left\{\lambda_{i,Z}\right\}$, where $\lambda_{i,Z}$ are simple eigenvalues for $-{\cal L}_Z$, furthermore, $\lambda_{i,Z}\to \lambda_{i,0}$, as $Z\to 0$.
\item Applying the Kato-Rellich's Theorem (Theorem XII.8 in \cite{Reed}) for each one of the simple eigenvalues $\lambda_{i,0}$, $i=1,2,$ we obtain the existence of a positive $Z_0<Z_1$, and analytic functions $\Pi_2, \Omega_2$ defined on the intervals $(-Z_0,Z_0)$ satisfying the items $(i),(ii)$ and $(iii)$ of the theorem.
\end{enumerate}
\end{proof}

Now, we proceed to count the number of negative eigenvalues of the operator $-{\cal L}_Z$. First, we assume that $Z$ is small.
\begin{Lemma}\label{lem6.8} 
There exists $r\in\mathcal{R}$, $0<r<Z_0$ such that $\Pi_2(Z)<0$, for any $Z\in(-r,0)$ and $\Pi_2(Z)>0$ for any $Z\in(0,r)$. Therefore, $-{\cal L}_{Z}$ has exactly two negative eigenvalues para $Z$ negative and small and $-{\cal L}_{Z}$ has exactly one negative eigenvalue for $Z$ positive and small.    
\end{Lemma}
\begin{proof}
By Taylor's Theorem around $Z=0$, the functions $\Pi_2$ and $\Omega_2$ in lemma \ref{piomega} can be written as
\begin{equation}\label{taylorpiomega}
\begin{aligned}
\Pi_2(Z)&=\beta Z+O\left(Z^2\right),\\
\Omega_2(Z)&=\phi'_{w,0}+Z\psi_0+O\left(Z^2\right),
\end{aligned}
\end{equation}
where $\phi'_{w,0}=\frac{d}{dx}\phi_{w,0}$, $\beta\in \mathbb{R}$, ($\beta=\Pi_2'(0)$) and $\psi_0\in L^2(\mathbb{R})$ $(\psi_0=\Omega_2'(0))$. To show the result, it is enough to show that $\beta>0$ or equivalently that $\Pi_2(Z)$ is an increasing function of the variable $Z$ around $Z=0$. Since the function $Z\rightarrow\phi_{w,Z}$ is analytic, then for $Z$ close to zero, we have that   
\begin{equation*}\label{taylorsol}
\phi_{w,Z}=\phi_{w,0}+Z\chi_0+O\left(Z^2\right),
\end{equation*}
where 
\begin{equation}\label{explicitchi}
\chi_0=\left.\frac{d}{dZ}\phi_{w,Z}\right|_{Z=0}.
\end{equation}
Now, from the equation (\ref{ltaordenada}), we have that for all $\psi\in H^1(\mathbb{R})$,
\begin{equation}\label{delemas}
\left\langle -\phi''_{w,Z}-w\phi_{w,Z}-a\phi^{p}_{w,Z}-b\phi^{2p-1}_{w,Z},\psi \right\rangle=Z\phi_{w,Z}(0)\psi(0).
\end{equation}Taking derivative with respect to the variable $Z$ in (\ref{delemas}) and evaluating in $Z=0$, we get that
\begin{equation}\label{delemaduro}
\left\langle -{\cal L}_{0}\chi_0,\psi\right\rangle=\phi_{w,0}(0)\psi(0).
\end{equation}In order to obtain $\beta$ as a function of the variable $Z$. We compute the quantity $\left\langle -{\cal L}_{Z}\Omega_2(Z),\phi'_{w,0}\right\rangle$ in two different ways.   
\begin{enumerate}[(1)]
\item Since $-{\cal L}_{Z}\Omega_2(Z)=\Pi_2(Z)\Omega_2(Z)$, then from (\ref{taylorpiomega})
\begin{equation}\label{primform}
\left\langle -{\cal L}_{Z}\Omega_2(Z),\phi'_{w,0}\right\rangle=\beta Z||\phi'_{w,0}||^2+O\left(Z^2\right).
\end{equation}
\item Since $-{\cal L}_{Z}$ is selfadjoint, then $\left\langle -{\cal L}_{Z}\Omega_2(Z),\phi'_{w,0}\right\rangle=\left\langle \Omega_2(Z),-{\cal L}_{Z}\phi'_{w,0}\right\rangle$. Now, since $-{\cal L}_{0}\phi'_{w,0}=0$, it follows from (\ref{explicitchi}) that  
\begin{equation}\label{ele}
\begin{aligned}
-{\cal L}_{Z}\phi'_{w,0}&=-{\cal L}_{0}\phi'_{w,0}+\left[f'\left(\phi_{w,0}\right)-f'\left(\phi_{w,Z}\right)\right]\phi'_{w,0}\\
&=\left[f'\left(\phi_{w,0}\right)-f'\left(\phi_{w,0}+Z\chi_0+O(Z^2)\right)\right]\phi'_{w,0}\\
&=-f''\left(\phi_{w,0}\right)Z\chi_0\phi'_{w,0}+O\left(Z^2\right),
\end{aligned}
\end{equation}where $f(x)=ax^p+bx^{2p-1}$. Thus, from (\ref{taylorpiomega}) and (\ref{ele}), we obtain that
\begin{equation}\label{masmas}
\begin{aligned}
\left\langle -{\cal L}_{Z}\Omega_2(Z),\phi'_{w,0}\right\rangle &=-Z\left\langle\phi'_{w,0},f''\left(\phi_{w,0}\right)\chi_0\phi'_{w,0}\right\rangle+O(Z^2).
\end{aligned}
\end{equation}
On the other hand, by direct computation, we see that 
\begin{equation}\label{tal}
-{\cal L}_{0}(-w\phi_{w,0}-f\left(\phi_{w,0}\right))=f''(\phi_{w,0})[\phi'_{w,0}]^2.
\end{equation}
Since $\phi$ satisfies the equation (\ref{ordendoisz0}), it follows from (\ref{delemaduro}), (\ref{masmas}), (\ref{tal})

\begin{equation}\label{fin}
\begin{aligned}
\left\langle -{\cal L}_{Z}\Omega_2(Z),\phi'_{w,0}\right\rangle&=-Z\left\langle\chi_0,f''(\phi_{w,0})[\phi'_{w,0}]^2\right\rangle+O\left(Z^2\right)\\
&=-Z\left\langle-{\cal L}_{0}\chi_0,-w\phi_{w,0}-f\left(\phi_{w,0}\right)\right\rangle+O\left(Z^2\right)\\
&=-Z\phi_{w,0}(0)(-w\phi_{w,0}(0)-f(\phi_{w,0}(0)))+O\left(Z^2\right)\\
&=-Z\phi_{w,0}(0)(\phi''_{w,0}(0))+O\left(Z^2\right).
\end{aligned}
\end{equation} 
\end{enumerate}
Finally, from (\ref{primform}) and (\ref{fin}), we conclude that
\begin{equation*}
\beta=\frac{\phi_{w,0}(0)(-\phi''_{w,0}(0))}{||\phi'_{w,0}||^2}+O(Z).
\end{equation*}
Hence $\Pi_2'[0]=\beta>0$. This completes the proof of the lemma. 
\end{proof}
Next, we extend the results in the Lemma \ref{lem6.8} for all the values of $Z$.   

\begin{Lemma}\label{Numberofnegeig}Let $Z\in\mathbb{R}$ and set $n\left(-{\cal L}_Z\right)$ the number of negative eigenvalues of $-{\cal L}_{Z}$, then we have that
\begin{enumerate}
\item For $Z<0$, $n\left(-{\cal L}_Z\right)=2$.
\item  For $Z\geq0$, $n\left(-{\cal L}_Z\right)=1$.
\end{enumerate}
\end{Lemma}
\begin{proof}
The proof is based in Lemma \ref{Kert} above and the ideas in Le Coz {\it et al.}  \cite{LeCoz}.
\end{proof}

Finally, the proof of the Theorem \ref{MainResult} follows from the Theorem 5.1.3 and 5.2.1 in \cite{DAN HENRY}.

\renewcommand{\refname}{References}

\end{document}